\documentclass[12pt]{amsart}
\numberwithin{equation}{section}
\usepackage{amsmath,amsthm,amsfonts,amscd,eucal,mathabx}

\usepackage{xypic}
\usepackage{graphicx}
\usepackage{amssymb}
\def\ca{{\mathcal A}}
\def\cb{{\mathcal B}}
\def\cc{{\mathcal C}}

\def\ce{{\mathcal E}}

\def\ck{{\mathcal K}}

\def\cs{{\mathcal S}}

\def\ga{{\mathfrak A}}

\def\bc{{\mathbb C}}

\def\bm{{\mathbb M}}
\def\bn{{\mathbb N}}

\def\bt{{\mathbb T}}

\def\bz{{\mathbb Z}}

\def\a{\alpha}
\def\b{\beta}
  
\def\d{\delta}

\def\l{\lambda} 

\def\m{\mu}

\def\n{\nu}
\def\r{\rho}
\def\s{\sigma} 

\def\f{\varphi}  \def\F{\Phi}
\def\th{\theta} 
\def\om{\omega}

\def\id{\hbox{id}}

\newtheorem{thm}{Theorem}[section]
\newtheorem{lem}[thm]{Lemma}

\newtheorem{cor}[thm]{Corollary}
\newtheorem{prop}[thm]{Proposition}
\newtheorem{rem}[thm]{Remark}
\newtheorem{defin}[thm]{Definition}

\theoremstyle{definition}
\newtheorem{examp}{Example}[section]

\def\aut{\mathop{\rm Aut}}

\def\tr{\mathop{\rm Tr}}

\def\supp{\mathop{\rm supp}}

\def\di{{\rm d}}

\def\idd{{1}\!\!{\rm I}}

\DeclareMathAlphabet{\mathpzc}{OT1}{pzc}{m}{it}

\begin{document}
\title[unique ergodicity]
{invariant conditional expectations and unique ergodicity for anzai skew-products}
\author[S. Del Vecchio]{Simone Del Vecchio}
\address{Simone Del Vecchio\\
Dipartimento di Matematica \\
Universit\`{a} di Roma Tor Vergata\\
Via della Ricerca Scientifica 1, Roma 00133, Italy}\email{{\tt
delvecchio@mat.uniroma2.it}}
\author[F. Fidaleo]{Francesco Fidaleo}
\address{Francesco Fidaleo\\
Dipartimento di Matematica \\
Universit\`{a} di Roma Tor Vergata\\
Via della Ricerca Scientifica 1, Roma 00133, Italy} \email{{\tt
fidaleo@mat.uniroma2.it}}
\author[S. Rossi]{Stefano Rossi}
\address{Stefano Rossi\\
Dipartimento di Matematica \\
Universit\`{a} degli Studi di Bari Aldo Moro \\
Via Edoardo Orabona 4, Bari 70125, Italy} \email{{\tt
stefano.rossi@uniba.it}}
%\date{\today}


\begin{abstract}
Anzai skew-products are shown to be uniquely ergodic with respect to the fixed-point subalgebra 
if and only if there is a unique conditional expectation onto such a subalgebra which is invariant
under the dynamics. For the particular case of skew-products, this solves a question raised by B. Abadie and K. Dykema in the wider context of $C^*$-dynamical
systems.

\end{abstract}

\maketitle

\section{introduction}

A $C^*$-dynamical system, {\it i.e.} a pair $(\ga, \Phi)$ given by a unital $C^*$-algebra $\ga$ with
unit $\idd_\ga$  and a unital $*$-automorphism 
$\Phi$ of $\ga$, is said to be uniquely ergodic when there exists exactly one $\Phi$-invariant state $\om$ on $\ga$.
This condition turns out to be equivalent to the seemingly stronger condition that, for every
$a\in\ga$, the Ces\`{a}ro averages
$\frac{1}{N}\sum_{k=0}^{N-1}\Phi^k(a)$ converge to $\om(a)\idd_\ga$ in norm.
As is known, either condition implies that the fixed-point subalgebra $\ga^\Phi:=\{a\in\ga\mid \Phi(a)=a\}$ is trivial, {\it i.e.}
$\ga^\Phi=\bc \idd_\ga$.

It is then natural to turn one's attention to
dynamical systems  $(\ga, \Phi)$ for which the fixed-point subalgebra is allowed  to be nontrivial 
but nevertheless any state on $\ga^\Phi$ only admits a unique
$\Phi$-invariant extension to the whole $\ga$. To our knowledge, systems of this type were first introduced by Abadie and Dykema in \cite{AD}, where they are
referred to as dynamical systems uniquely ergodic with respect to the fixed-point subalgebra, in that they are a broad
generalization of uniquely ergodic systems.
Among other things, in that paper a number of equivalent conditions are given for a system to be uniquely ergodic with
respect to the fixed point subalgebra.
For instance, one is that, for any $a\in\ga$,  the Ces\`{a}ro averages $\frac{1}{N}\sum_{k=0}^{N-1}\Phi^k(a)$ still converge to some necessarily $\Phi$-invariant
element $E(a)$. Note that $E$ defines a conditional expectation of $\ga$ onto $\ga^\F$ that is $\F$-invariant, namely
$E\circ\F=E$. Moreover, $E$ is actually the only such conditional expectation.

It is  not known, however, whether the existence of a unique
$\F$-invariant conditional expectation as above is enough to obtain
unique ergodicity with respect to the fixed-point subalgebra, and this was indeed formulated as a question in \cite{AD}.
The present work settles the problem for so-called skew products, which are a remarkable family of
classical dynamical systems, namely the $C^*$-algebra is commutative with the underlying topological space being a product of the form $X_o\times\bt$, where $X_o$ is any compact Hausdorff space, and the dynamic is assigned
through a homeomorphism $\Phi_{\theta_o, f}$ acting on $X_o\times\bt$ as 
$\Phi_{\theta_o, f}(x, z):=(\theta_o(x), f(x)z)$, where $\theta_o$ is a uniquely ergodic homeomorphism of
$X_o$ and $f:X_o\rightarrow \bt$ is  a continuous function.

We prove that any such  system is uniquely ergodic with respect to the fixed point-subalgebra if and only if
there exists a unique conditional expectation onto the fixed point-subalgebra. To accomplish this goal, we
make use of another characterization of unique ergodicity with respect to the fixed-point algebra for skew products which has been 
proved in our previous work
\cite{DFR}. Indeed,  a skew product $(X_o\times\bt, \Phi_{\theta_o, f})$ is there seen to be  uniquely ergodic with respect to the fixed point algebra if and only if,
for any $n\in\bz$, a function $g: X_o\rightarrow \bc$ satisfying $g(\theta_o(x))f^n(x)=g(x)$ is provided
by a continuous (possibly zero) function. Our strategy, therefore, will  be to show that when non-continuous solutions do exist, it is
always possible to exhibit uncountably many $\Phi_{\theta_o, f}$-invariant conditional expectations.

Quite interestingly, the analysis of the invariant conditional expectations  can  be pushed further, for
we also show that all invariant conditional expectations are dominated by a distinguished expectation as long as
the fixed-point subalgebra is not trivial. Furthermore, this conditional expectation is exactly the only invariant conditional expectation when the skew product
is uniquely ergodic with respect to the fixed-point subalgebra. Ultimately these facts allow us to spell out an extension of Fustenberg's characterization of uniquely ergodic skew-product dynamical systems
 (\cite{Fu}, Lemma 2.1) to the case of uniquely ergodic systems with respect to the fixed point subalgebra in terms of invariant conditional expectations (Theorem \ref{abadyk}).

\section{preliminaries}

A (discrete) $C^*$-dynamical system is a pair $(\ga,\F)$ made of a $C^*$-algebra and positive map $\F:\ga\to\ga$. Suppose that $\ga$ is unital with identity $\idd:=\idd_\ga$, and $\F$  completely positive and unital, that is $\F(\idd)=\idd$. It is said to be {\it topologically ergodic} if $\ga^\F=\bc\idd$ for the fixed-point subspace $\ga^\F:=\{a\in\ga: \F(a)=a\}$.

The set $\cs(\ga)^\F:=\{\f\in\cs(\ga)\mid \f\circ\F=\f\}$ of the invariant states is convex and weak-${}^*$ compact. The extremal invariant states are said to be {\it ergodic}. 
If the set of the invariant states is a singleton, that is $\cs(\ga)^\F=\{\f\}$, the $C^*$-dynamical system $(\ga,\F)$ is said to be {\it uniquely ergodic}. If in addition $\f$ is faithful, it is said to be {\it strictly ergodic}. In the uniquely ergodic case, we have $\ga^\F=\bc\idd$,
$E(\,{\bf\cdot}\,):=\f(\,{\bf\cdot}\,)\idd$ is an invariant completely positive projection onto the fixed-point subspace, and for the Ces\'aro averages,
$$
\lim_n\frac{1}{n}\sum_{k=0}^{n-1}\Phi^k(a)=E(a)\,,\quad a\in\ga\,,
$$
in norm. 

\medskip

The notion of unique ergodicity was recently generalised to the case when the fixed-point subspace is non trivial. The reader is referred to \cite{AD}, Definition 3.3,
for $*$-automorphisms where $\ga^\F$ is a $C^*$-subalgebra, and \cite{FM3}, Definition 2.2, for the more general case of completely positive maps.\footnote{In \cite{FM3}, Theorem 2.1,
it was also shown that the a-priori weaker condition (v), also characterises the unique ergodicity w.r.t. the fixed-point subspace.} For the purpose of the present paper, we adopt the following definition of unique ergodicity w.r.t. the fixed-point subalgebra.
\begin{defin}
\label{ADFM}
A $C^*$-dynamical system $(\ga,\a)$, with $\ga$ unital and $\a\in\aut(\ga)$ a $*$-automorphism, is said to be uniquely ergodic w.r.t. the fixed point subalgebra if the sequence 
$\left(\frac{1}{n}\sum_{k=0}^{n-1}\a^k(a)\right)_n$ converges in norm for each $a\in\ga$.
\end{defin}
With an abuse of notation, we say that the automorphism itself $\a$ is uniquely ergodic if it causes no confusion.

For uniquely ergodic systems as in Definition \ref{ADFM}, the limit of the Ces\'aro averages defines a $\a$-invariant conditional expectation $E:\ga\to\ga^\a$ onto the fixed-point subalgebra given by
$$
E(a):=\lim_n\frac{1}{n}\sum_{k=0}^{n-1}\a^k(a)\,,\quad a\in\ga\,,
$$
which is necessarily unique. 

Therefore, if $\ga^\a=\bc\idd$, the unique ergodicity ({\it i.e.} the convergence in norm of all averages $\left(\frac{1}{n}\sum_{k=0}^{n-1}\a^k(a)\right)_n$) is equivalent to the existence of a unique invariant conditional expectation $E:\ga\to\ga^\a$ which, due to the triviality of the fixed-point subalgebra, leads to $E=\f(\,{\bf\cdot}\,)\idd$, $\f\in\cs(\ga)$ being the unique invariant state.

It is of certain interest to decide when the question raised by B. Abadie and K. Dykema ({\it cf.} Question 3.4) of whether the unique ergodicity is equivalent to the existence of a unique invariant conditional expectations holds true. However, for a wide class of Anzai skew-products, called also {\it processes on the torus} in \cite{Fu}, for which $\ga^\a$ is always either trivial or infinite dimensional, we will show that the existence of a unique invariant conditional expectation onto the fixed-point subalgebra is indeed equivalent to unique ergodicity w.r.t. the fixed-point subalgebra.
\begin{rem}
\label{uefpsb}
For the case $\ga^\a=\bc\idd$, Definition \ref{ADFM} is equivalent to the usual one: $(\ga,\a)$ is uniquely ergodic if, by definition, $\cs(\ga)^\a$ is a singleton, see {\it e.g.} \cite{vW}, Theorem 4.1.8.
\end{rem}

\medskip

From now on, we suppose that $\ga$ is a unital abelian $C^*$-algebra, and $\a$ is a $*$-automorphism. It is well known that any such a $C^*$-dynamical system arises as follows, $\ga\sim C(X)$, $X\sim\s(\ga)$ being a compact Hausdorff space uniquely determined up to topological isomorphisms, and $\a(f):=f\circ\th_o$ for some $\theta_o\in\textrm{Homeo}(X)$, the space of all homeomorphisms of $X$. With a slight abuse of notation, we denote any such a $C^*$-dynamical system as above with $(X,\th_o)$, and call these simply  "a dynamical system". 

One of such dynamical systems is thus uniquely ergodic if, by definition, there is only one regular Borel probability measure $\m$ ({\it i.e.} a positive normalised Radon measure) on $X$ which is invariant under the transposed action 
$\n\to\n\circ\th_o^{-1}$ of $\th_o$ induced on measures $\n$, $\m=\m\circ\th_o^{-1}$. It is strictly ergodic if, in addition, $\supp(\m)=X$.

A triplet $(X,\th_o,\m)$ denotes also a dynamical system, as soon as we want to point out any invariant measure $\m$ as above, in particular when $(X,\th_o)$ is uniquely ergodic and $\m$ is its unique invariant measure.

\medskip

The dynamical systems with which we deal with, called in \cite{Fu} {\it processes on the torus}, are those on the cartesian product $(X_o\times\bt)$, where $X_o$ is a compact Hausdorff space and $\bt$ is the one dimensional torus.   

On $X_o\times\bt$, for each $n>1$ we consider the periodic homeomorphism $\id_{X_o}\!\times\!R_{2\pi\imath/n}$, together with the fixed-point subalgebra 
$$
C(X_o\times\bt)^{\b_n}=\overline{\bigg\{\sum_{l\in F}g_l(x)z^{ln_o}\mid g_l(x)\in C(X_o), F\,\text{finite subset of}\,\, \bz\bigg\}}
$$
w.r.t. the canonical (dual) action 
$\b_n$ on $C(X_o\times\bt)$ associated to such an homeomorphism.

A canonical $\b_n$-invariant conditional expectation onto $C(X_o\times\bt)^{\b_n}$ is uniquely defined by its action on generators
\begin{equation}
\label{perce}
\ce_n\big(h(x)z^k\big):=h(x)z^{ln}\d_{k,ln}\,,\quad k,l\in\bz\,.
\end{equation}
Indeed, since
$$
\ce_n(f)=\frac1{n}\sum_{l=0}^{{n}-1}\b_n^l(f)\,,
$$
we deduce that $\ce_n$ is a faithful conditional expectation of $C(X_o\times\bt)$ onto $C(X_o\times\bt)^{\b_n}$ which is invariant under the action of $\b_n$.\footnote{The case $n=1$ corresponds to the trivial homeomorphism $\id_{(X_o\times\bt)}$ leading to the trivial fixed-point subalgebra $C(X_o\times\bt)$ and trivial conditional expectation 
$\ce_1=\id_{C(X_o\times\bt)}$.}

Our starting point will be a uniquely ergodic dynamical system $(X_o,\th_o,\m_o)$. Corresponding to a given continuous function $f\in C(X_o;\mathbb{T})$, we consider the {\it Anzai skew-product} ({\it cf.} \cite{A})
$\Phi_{\theta_o, f}\in\textrm{Homeo}(X_o\times\mathbb{T})$ defined as
\begin{equation}
\label{askp}
\Phi_{\theta_o, f}(x, z):=\big(\theta_o(x), f(x)z\big)\,,\quad (x, z)\in X_o\times\mathbb{T}\,.
\end{equation}
It is seen in \cite{Fu} that the product measure $\mu:=\mu_o\times m$, where 
$$
m=\frac{\di\th}{2\pi}=\frac{\di z}{2\pi\imath z}\,,\quad z=e^{\imath \th}\in\bt\,,
$$
is the Haar-Lebesgue measure 
of the unit circle $\bt$, is invariant
for the dynamics induced by $\Phi_{\theta_o, f}$ on $(X_o\times \bt)$.

\medskip

Most of the ergodic properties of $(X_o\times\bt, \F_{\theta_o, f})$
can  be read through the kind of the solutions of the so-called cohomological equations, one for each $n\in\bz$.\footnote{The case $n=0$ corresponds always to the trivial solution $f(x)=1$, $\m_o$ a.e., up to a multiplicative constant.}

More precisely, for $g\in L^\infty(X_o, \mu_o)$ consider the multiplication operator $M_g\in\cb\big(L^2(X_o, \mu_o)\big)$ given by
$$
(M_g\xi)(x):=g(x)\xi(x)\,,\quad \xi\in L^2(X_o, \mu_o)\,.
$$
We also have a (cyclic, with cyclic vector $\xi(x):=1$, $\m_o$ a.e.) representation $\pi_{\m_o}$ of $C(X_o)$ by multiplication operators, given for $G\in C(X_o)$ as 
$$
(\pi_{\m_o}(G)\xi)(x):=G(x)\xi(x)\,,\quad \xi\in L^2(X_o, \mu_o)\,.
$$

Corresponding to a skew-product $\F_{\theta_o, f}$, for each $n\in\bz$ we consider the {\it cohomological equations}
\begin{equation}
\label{sifo}
g(\theta_o(x))f(x)^n=g(x)\,,\,\,\mu_o\,\text{-}\,\textrm{a.e.}\,,
\end{equation}
in the unknown complex function $g\in L^\infty(X_o, \mu_o)$. We also consider the twin equation
\begin{equation}
\label{sifo1}
g(\theta_o(x))f(x)^n=g(x)\,,
\end{equation}
where the unknown is now a function $g\in C(X_o)$. Obviously if $G$ satisfies \eqref{sifo1}, $\pi_{\m_o}(G)$ satisfies \eqref{sifo}.

For a fixed $n$, it is therefore natural to say that solutions $g=\pi_{\m_o}(G)$ with $G$ satisfying \eqref{sifo1} are said to be {\it continuous}, whereas the remaining one are named
{\it measurable non-continuous}.

By unique ergodicity of $\theta_o$, the equation only has constant solutions for $n=0$. Also note that, for every $n\in\bz$,
the function that is zero is is a solution of \eqref{sifo1}, whereas that which is zero $\mu_o$-a.e. is a solution of \eqref{sifo}. We
we shall refer to those as the trivial solutions of the cohomological equation.

Throughout the paper, if a nontrivial solution of \eqref{sifo} is continuous 
(up to being re-defined on a $\mu_o$-negligible set) and satisfies
\eqref{sifo1}, we will simply say that the cohomological equations have non-trivial
continuous solutions.
Note also that, if $g$ is a solution of \eqref{sifo} at the level $n$, then the two-variable function
$h(x, z):= g(x)z^n$ is a continuous $\F_{\theta_o, f}$-invariant function if and only if $g$ is continuous and
satisfies \eqref{sifo1}.\footnote{To be more precise, whenever $g$ is a solution of \eqref{sifo} at the level $n$, $h(x, z):= g(x)z^n$ is $\m_o\times m$-equivalent to a 
$\F_{\theta_o, f}$-invariant continuous function if and only if $g=\pi_{\m_o}(G)$, and $G$ satisfies the twin equation \eqref{sifo1}.}

We remark that, due to ergodicity of $(X_o, \theta_o, \mu_0)$, the solution of \eqref{sifo} for a fixed
$n\in\bz$ is unique up to a multiplicative scalar. This was seen in \cite{DFR}, Proposition 8.2,  by adapting the proof
of \cite{F20}, Proposition 2.2, to the present situation. Moreover, there is no loss of generality if those solutions are multiple of a single function with absolute value $1$, almost everywhere w.r.t. $\mu_o$, see \cite{DFGR}, Remark 4.2.

\medskip

In \cite{Fu}, it was proved that the system $(X_o\times\bt, \F_{\theta_o, f}, \mu)$ is ergodic if and only if, for every $n\neq 0$, 
\eqref{sifo} only have the trivial solution. Remarkably, the system $(X_o\times\bt, \Phi_{\theta_o, f})$ is uniquely ergodic
if and only if $(X_o\times\bt, \Phi_{\theta_o, f}, \mu)$ is ergodic, provided that $(X_o, \theta_o)$ is uniquely ergodic
with $\mu_o$ the unique invariant measure.
In addition, topological ergodicity of
$(X_o\times\bt, \F_{\theta_o, f})$, that is $h\in C(X_o\times\bt)$ with $h\circ\F_{\theta_o, f}=h$ implies that $f$ is constant,
is equivalent to the weaker request that continuous solutions of the cohomological equations are null for each $n\neq0$.

\medskip

The analysis in \cite{Fu} for processes on the torus leaves open the case when the fixed-point subalgebra is non trivial, that is unique ergodicity w.r.t. the fixed-point subalgebra, which has been recently addressed in \cite{DFR}. Indeed, in Theorem 10.7 of this paper it has been proved to
amount to the condition that, for any $n\in\bz$, any solution of \eqref{sifo}
is automatically continuous, that is its class of $\mu_o$-equivalence contains a continuous function 
satisfying \eqref{sifo1}.

The following simple result helps to further clarify the relation between solutions of cohomological equations \eqref{sifo} and \eqref{sifo1}.

\begin{prop}
If $g\in C(X)$ satisfies \eqref{sifo} for $n\in\bz$, then $g(\theta_o(x))f(x)^n=g(x)$ for $x\in\supp(\m_o)$, and therefore $g$ satisfies automatically \eqref{sifo1} if $(X,\th_o)$ is strictly ergodic.
\end{prop}
\begin{proof}
Suppose that If $g\in C(X)$ satisfies \eqref{sifo} and choose any Borel set $A\subset X$ of full measure and, necessarily, $A\bigcap\supp(\m_o)$ is dense in $\supp(\m_o)$. For each $x\in\supp(\m_o)$, choose a net $(x_\iota)_\iota\subset A\bigcap\supp(\m_o)$ converging to $x$. We get
\begin{align*}
g(x)=&g\big(\lim_\iota x_\iota\big)=\lim_\iota g(x_\iota)=\lim_\iota\big(g(\theta_o(x_\iota))f(x_\iota)^n\big)\\
=&g\big(\theta_o(\lim_\iota x_\iota)\big)f(\lim_\iota x_\iota)^n=g(\theta_o(x))f(x)^n\,.
\end{align*}
\end{proof}

\begin{rem}
If $g\in C(X)$ satisfies \eqref{sifo} for $n\in\bz$, then $g(\theta_o(x))f(x)^n=g(x)$ for $x\in\supp(\m_o)$, and therefore $g$ satisfies automatically \eqref{sifo1} if $(X,\th_o)$ is strictly ergodic.
\end{rem}
\begin{proof}
Indeed, suppose that If $g\in C(X)$ satisfies \eqref{sifo} and choose any Borel set $A\subset X$ of full measure and, necessarily, $A\bigcap\supp(\m_o)$ is dense in $\supp(\m_o)$. For each $x\in\supp(\m_o)$, choose a net $(x_\iota)_\iota\subset A\bigcap\supp(\m_o)$ converging to $x$. We get
\begin{align*}
g(x)=&g\big(\lim_\iota x_\iota\big)=\lim_\iota g(x_\iota)=\lim_\iota\big(g(\theta_o(x_\iota))f(x_\iota)^n\big)\\
=&g\big(\theta_o(\lim_\iota x_\iota)\big)f(\lim_\iota x_\iota)^n=g(\theta_o(x))f(x)^n\,.
\end{align*}
\end{proof}

We end the section by remarking that all sums arising from the Fourier analysis on the unit circle $\bt$ are understood to converge in the sense of Ces\'aro w.r.t. a fixed topology, usually that generated by the norm if is not differently specified, see {\it e.g.} \cite{DFR},

\section{on invariant conditional expectations}

We start with some results which are useful in the sequel. The first one provides a parametric generalisation of the Fej\'er-Riesz Theorem which has a self-containing interest.
\begin{prop}
\label{wcrf}
Let $X_o$ be a topological  space together with a strictly positive trigonometric polynomial 
$$
p_K(x,z):=\sum_{|k|\leq K} b_k(x)z^k,\quad (x, z)\in X_o\times\bt
$$
({\it i.e.} $p_K(x,z)> 0$ for every $(x, z)\in X_o\times\bt$), where the coefficients $b_k$, $k=-K, \ldots, K$, are complex-valued (bounded) Borel functions on $X_o$. 

Then there exists a trigonometric polynomial $g_K(x,z)=\sum_{k=0}^K a_k(x)z^k$, where the coefficients $a_k$, $k=0, \ldots, K$ are  
(bounded)  Borel functions, such that $p_K=\overline{g_K}g_K$.
\end{prop}
\begin{proof}
We start by recalling that the Fej\'er-Riesz Theorem (see {\it e.g.} \cite{P}, Lemma 2.5) gives an explicit formula for the square root of a strictly positive trigonometric polynomial $q_K(z)=\sum_{|k|\leq K}b_k z^k$ with complex coefficients in terms of the roots of the polynomial $z^K q_K(z)$. 

More precisely, let $q_K(z)=\sum_{|k|\leq K}b_k z^k$ with $b_K\neq 0$, be positive for $z\in\bt$, and consider the $K$-roots $z_1,\cdots, z_K$ (counted with their multiplicity) of the polynomial $z^K q_K(z)$ which lie in the complement of the unit disk $\{z\in\mathbb{C}: |z|>1\}$.
Then the trigonometric polynomial 
\begin{equation}
\label{fejer}
g_K(z)=\left|\frac{b_K}{z_1\cdots z_K}\right|^{1/2}\prod_{i=1}^K (z-z_i)
\end{equation}
satisfies the required property: $|g_K(z)|^2=q_K(z)$, $z\in \bt$.
We now start by handling the case where $b_K(x)\neq 0$ for every $x\in X_o$.

Define $C:=\{(w_0, w_1, \ldots, w_{2K})\in\bc^{2K+1}\mid w_{2K}=0\}\subset \bc^{2K+1}$.
Consider the map 
$$
\Sigma: \mathbb{C}^{2K+1}\setminus C \rightarrow \mathbb{C}^{2K}/S^{2K}
$$
that, to the $2K+1$-tuple $(w_0, w_1, \ldots, w_{2K})$, associates the set of the $2K$ roots
of the polynomial $p(z)=\sum_{j=0}^{2K}w_jz^j$ considered as an element of the quotient
of $\bc^{2K}$ by the natural action of the symmetric group $S^{2K}$.
By \cite{HM}, Theorem A, the map $\sigma$ is  continuous.

Given $q_K(z)=\sum_{|k|\leq K}b_kz^k$, 
we denote by $w_j$, $j=0, 1, \ldots, 2K$  the coefficients of $z^Kq_K(z)$, that is
$\sum_{j=0}^{2K}w_jz^j:=z^Kq_K(z)$.
We note that
$$
D:=\bigg\{(w_0, w_1, \ldots, w_{2K})\in \bc^{2K+1}\setminus C\,\Big| \sum_{j=0}^{2K} w_j z^{j-K}>0,\,\,\, z\in\bt \bigg\}.
$$ 
is a Borel subset of $\bc^{2K+1}$. 
Indeed, $D=(C^{2K+1}\setminus C)\cap D_1\cap D_2$, 
where 
$$
D_1:=\bigg\{(w_0, w_1, \ldots, w_{2K})\in \bc^{2K+1}\,\Big|\,{\rm Re}\sum_{j=0}^{2K} w_j z^{j-K}>0,\,\,\, z\in\bt \bigg\}
$$
is open, 
and 
$$
D_2:=\bigg\{(w_0, w_1, \ldots, w_{2K})\in \bc^{2K+1}\,\Big|\,{\rm Im}\sum_{j=0}^{2K} w_j z^{j-K}=0,\,\,\, z\in\bt \bigg\}
$$
is closed.

Consider the restriction of $\Sigma$ to $D$, and note that 
$\Sigma(D)$ is contained in the set of those non ordered $2K$-tuples such that $K$ entries have
absolute value strictly greater than $1$ and $K$ entries have absolute value strictly less than $1$.

Define $\pi:=P\circ\Sigma$ from $D$ to $\bc^K/S^K$, where $P:\Sigma(D)\rightarrow \bc^K/S^K$ is the map that selects the $K$ entries 
whose absolute value is greater than $1$. The map $\pi$ is continuous on $D$ as it is the composition of continuous 
maps. 

Now, from \eqref{fejer} one sees  that, under our hypotheses, the coefficients
$a_k$, $k=0, \ldots, K$, are Borel functions on $X_o$. Indeed, these are obtained as  continuous symmetric 
functions of the roots
$z_1, z_2, \ldots, z_K$, which are in turn
measurable functions on $X_o$ since they are the composition of the Borel measurable map
$$
X_o\ni x\rightarrow (b_{-K}(x), \ldots, b_K(x))\in D^{2K+1}
$$
with the continuous map $\pi$.

The general case can be dealt with by defining recursively for $l=0, 1, \ldots, K-1$,
$$
M_{K-l}:=\{x\in X_o\mid b_j(x)=0,\, j=K, \ldots, K-l+1, \, {\rm and}\, c_{K-l}(x)\neq 0\}\,,
$$
which are Borel subsets of $X_o$.

The proof ends by employing the same technique as above on the subsets $M_{k-l}$, $l=0,1, \ldots, K-1$ to obtain the coefficients
of $g_K$, by gluing  finitely many Borel functions.

\bigskip

Concerning the boundedness of the coefficients $a_l$, we easily have
$$
\big|g_K(x,z)\big|^2=p_K(x,z)\leq\sup_{(x,z)\in X_o\times\bt} p_K(x,z)\,,\quad (x,z)\in X_o\times\bt\,.
$$
Therefore, for each $x\in X_o$,
$$
|a_l(x)|=\left|\oint g_K(x,z)z^{-l}\frac{\di z}{2\pi\imath z}\right|\leq\sqrt{\sup_{(x,z)\in X_o\times\bt} p_K(x,z)}\,,
$$
which leads to
$|g_K(x,z)\big|\leq (K+1)\sqrt{\sup_{(x,z)\in X_o\times\bt} p_K(x,z)}$\,.
\end{proof}
\begin{prop}
\label{iso0}
Suppose $U\in\ga$ is a unitary in the unital $C^*$-algebra $\ga$ such that, if $(\l_k)_{k\in\bz}\subset\ell^1(\bz)$ and 
$$
\sum_{k\in\bz}\l_k U^k=0\Rightarrow \l_k=0\,,\,\,\,k\in\bz\,.
$$ 
Then $C^*(U,\idd_\ga)\sim C(\bt)$ through the $*$-isomorphism that sends $U^n$ to the character $\chi_n(z):=z^n$, $n\in\bz$.
\end{prop}
\begin{proof}
The same argument employed in the proof of Proposition 10.8 in \cite{DFR}, which we report for the convenience of the reader.

Put
$$
\ga_o:=\bigg\{\sum_{m\in\mathbb{Z}} c_lU^l\mid\sum_{m\in\mathbb{Z}} |c_l|<\infty\bigg\}\,,
$$
endowed with the $\ell_1$-type norm
$$
\bigg[\!\bigg]\sum_{l\in\mathbb{Z}}c_lU^l\bigg[\!\bigg]:=\sum_{z\in\mathbb{Z}}|c_l|\,,
$$
and observe that any element in $\ga_o$ provides a well-defined element of $\ga$ because the above sums defining the elements of $\ga_o$ are absolutely convergent in $\ga$. 

In addition, $\ga_o$ is seen at once to be isometrically isomorphic with the Banach algebra 
$\ell_1(\mathbb{Z})$ understood
as the convolution algebra of $\mathbb{Z}$. 

Since the latter has only one $C^*$-completion, that is $C(\mathbb{T})$, we end the proof.
\end{proof}
\begin{examp}\label{exmat}
For the integer $k\geq1$, consider the unitary in $M_k(C(\bt))$ given by
\begin{equation*}
U_k=
\begin{pmatrix}
0 & 0 & \cdots & z \\
1 & 0 & \cdots & 0 \\
\vdots  & \vdots  & \ddots & \vdots  \\
0 & \cdots & 1 & 0
\end{pmatrix}\,.
\end{equation*}
Then $U_k^k$ is the diagonal matrix $zI_k=z\idd_{\bm_k(\bc)}$. 

It easy to check that Proposition \ref{iso} tells us that 
$$
C(\bt)\ni z^l\mapsto U_k^l\in\bm_k\big(C(\bt)\big)
$$ 
realises a $*$-monomorphism denoted by $\pi_k$.\footnote{The case $k=1$ corresponds to the trivial case when $M_k(C(\bt))=C(\bt)$ and $\pi_k=\id$.}
\end{examp}
For the skew-product in \eqref{askp}, we recall some properties of the solutions of the cohomological equations introduced above.
\begin{prop}
\label{nzmz}
The elements of $\bz$ for which \eqref{sifo} (resp. \eqref{sifo1}) admits nontrivial solutions provide a subgroup, and therefore there is an integer $n_{o}\geq0$ (resp. $m_{0}\geq0$) for which such a subgroup is given by $\{ln_{o}\mid l\in\bz\}$  (resp. $\{lm_{o}\mid l\in\bz\}$). Since for a fixed $n\in\bz$, if $g\in C(X)$ satisfies \eqref{sifo1}, it satisfies \eqref{sifo}, there is an integer $k_{o}\geq0$ such that $m_o=k_on_o$.

In addition, choosing a unitary $u_{n_{o}}$ satisfying \eqref{sifo} for $n=n_o$ (resp. satisfying \eqref{sifo1} for $n=m_o$), all solutions of \eqref{sifo} for $n=ln_o$ (resp. \eqref{sifo1}) $n=lm_o$) are a multiple of the powers of $u_{n_{o}}^{l}$ (resp. $u_{m_o}^l$). 
\end{prop}
\begin{proof}
The proof is analogous to that of \cite{DFR}, Proposition 10.2.
\end{proof}
\begin{rem}\label{k_o}
According to the results of Furstenberg, $n_o=0$ corresponds to unique ergodicity. The case $k_o=0$ corresponds to topological ergodicity, and finally $k_o=1$ to unique ergodicity w.r.t. the fixed-point subalgebra thanks to Proposition 10.7 in \cite{DFR}.
\end{rem}
From now on, we suppose that $n_o>0$ if is not otherwise specified, where $n_o$ is defined in Proposition \ref{nzmz}.

For $k\geq1$, denote by $\ca_k\subset L^\infty(X_o\times \bt)$ the $C^*$-algebra
generated by the functions $a_{lk}$ with 
\begin{equation}\label{characters}
a_{n}(x, z):=\big(u_{n_0}(x)z^{n_o}\big)^n\,,\quad n\in\bz\,.
\end{equation}
Obviously, all functions of $\mathcal{A}_k$ are $\Phi_{\th_o,f}$-invariant, $\mu$-a.e.\,\,.

We report the following facts which are direct consequences of Proposition \ref{iso0} and the Fej\'er-Riesz Theorem, respectively.
\begin{cor}
\label{iso}
For each integer $k\geq1$, $\ca_k\sim C(\bt)$ in the $*$-isomorphism that sends $\big(u_{n_0}\chi_{n_o}\big)^{kl}$ to the character $\chi_l(z):=z^l$, $l\in\bz$.
\end{cor}
\begin{proof}
It will follow from Proposition \ref{iso0}, once we have
verified its hypothesis is satisfied. 

To this aim, let $(\lambda_l)_{l\in\bz}\in \ell^1(\bz)$, such that $\sum_{l\in\bz} \lambda_l\big(u_{n_o}\chi_{n_o}\big)^{kl}=0$.
Note that the above series converges totally.

For every $l'\in\bz$, multiplying both members of the above series by $\big(u_{n_0}\chi_{n_o}\big)^{-kl}$ one has 
$$
0=\sum_{l\in\bz} \lambda_l\big(u_{n_o}\chi_{n_o}\big)^{kl}\big(u_{n_o}\chi_{n_o}\big)^{-kl'}=\sum_{l\in\bz} \lambda_l u_{n_o}^{k(l-l')}  \chi_{n_o}^{k(l-l')}\,.
$$
Integrating both members against the product measure $\di\m_o\times\di m$, exchanging the integral with the sum and finally applying Fubini Theorem, we get
\begin{align*}
0=&\int_{X_o\times\bt}\bigg(\sum_{l\in\bz} \lambda_l u_{n_o}(x)^{k(l-l')} \chi_{n_o}(z)^{k(l-l')}\bigg)\di\m_o(x)\times\di m(z)\\
=&\sum_{l\in\bz}\lambda_l \int_{X_o}u_{n_o}^{k(l-l')}(x){\rm d}\mu_o(x)\oint z^{kn_o(l-l')}\frac{\di z}{2\pi\imath z}\\
=&\sum_{l\in\bz}\lambda_l\d_{l,l'}=\lambda_{l'}\,,
\end{align*}
which concludes the proof.
\end{proof}
\begin{rem} 
Notice that, if $\ca_k\ni b=\sum_{|l|\leq L}b_l\big(u_{n_0}(x)z^{n_o}\big)^{kl}$ is positive, that is $b=c^*c$ for some $c\in\ca_k$, then there exists 
$a=\sum_{l=0}^La_l\big(u_{n_0}(x)z^{n_o}\big)^{kl}$ such that $b=a^*a$.
\end{rem}
\begin{proof}
It easily follows by Corollary \ref{iso} and Fej\'er-Riesz Theorem.
\end{proof}
We denote by $\r_k:\ca_k\rightarrow C(\bt)$ the isomorphisms described by Corollary \ref{iso}. We also note that $C(X_o\times\bt)^{\F_{\theta_o, f}}$ is isomorphic with 
$C(\bt)$ whenever $k_o>0$, see also  \cite{DFR}, Proposition 10.8. For $x\in C(X_o)$, $z\in\bt$ and $l\in\bz$, we indeed denote by $\sigma$ the map given by
\begin{equation}
\label{sigma}
C(\bt)\ni\chi_l\mapsto\s(\chi_l)(x,z):=\big(v_{n_o}(x)z^{n_o}\big)^{k_ol}\in C(X_o\times\bt)^{\F_{\theta_o, f}}\,,
\end{equation}
where $v_{n_ok_o}\in C(X_o)$ is a unitary satisfying \eqref{sifo1} for $n=n_ok_o$.

Therefore, the fixed-point subalgebra  $C(X_o\times\bt)^{\F_{\theta_o, f}}$ turns out to be isomorphic with $\ca_{k_o}$ under 
$\rho_{k_o}^{-1}\circ \sigma^{-1}$, which is nothing but
the map that sends each function in $C(X_o\times\bt)^{\F_{\theta_o, f}}$ to its equivalence class
in $L^\infty(X_o\times\bt, \mu)$.

\medskip

We are going to construct $\F_{\theta_o, f}$-invariant conditional expectations from
$C(X_o\times\bt)$ onto the fixed-point subalgebra $C(X_o\times\bt)^{\F_{\theta_o, f}}$. For such a purpose, we next single out a canonical contractive linear map $T$ from
$C(X_o\times\bt)$ to $\mathcal{A}_1$ which is also $\F_{\theta_o, f}$-invariant. To do this, we start by recalling that $h\in C(X_o\times\bt)$ can be expressed as a series
$h(x, z)=\sum_{n\in\bz} h_n(x)z^n$, with 
$h_n(x)=\oint h(x, z)z^{-n}\frac{\di z}{2\pi\imath z}$. 
Thanks to Fej\'er's theorem, the convergence of the series  holds in norm in the Ces\`{a}ro sense.
\begin{prop}
\label{quasiexp}
For $h(x, z)=\sum_{n\in\bz} h_n(x)z^n$,
$$
T(h):=\sum_{l\in\bz} a_{l}\int_{X_o}h_{ln_0}(x)u_{n_0}^{-l}(x){\rm d}\mu_o(x)
$$
with the convergence being understood in norm in the Ces\`{a}ro sense,
defines a linear contractive map of $C(X_o\times\bt)$ to $\mathcal{A}_1$.

Moreover, for $h\in C(X_o\times\bt)$ and $g\in C(X_o\times\bt)^{\Phi_{\th_o,f}}$,
$T(h\circ\Phi_{\th_o,f})=T(h)$, and 
$$
T(gh)=[g]_\mu T(h)\,,\quad T(hg)=T(h)[g]_\mu\,,
$$
where
$[g]_\mu$ denotes the equivalence class of $g$ in $L^\infty(X_o\times\bt, \mu)$.
\end{prop}
\begin{proof}
On $X_o\times\bt$, we consider the periodic homeomorphism $\id\!\times\!R_{2\pi\imath/n_o}$, together with the fixed-point subalgebra 
$$
C(X_o\times\bt)^{\b_{n_o}}=\bigg\{\sum_{l\in\bz}g_l(x)z^{ln_o}\mid g_l(x)\in C(X_o)\bigg\}
$$
w.r.t. the dual action associated to such an homeomorphism. A canonical $\b_{n_o}$-invariant faithful conditional expectation $\ce_{n_o}$ onto $C(X_o\times\bt)^{\b_{n_o}}$ is given 
\eqref{perce}.

We now claim that  $C(X_o\times\bt)^{\b_{n_o}}$ embeds  into the tensor product
$\mathcal{A}_1\otimes L^\infty(X_o, \mu_o)$ through the map $\iota$, completely determined on generators by
$$
\iota(h(x)z^{ln_o})= a_{l}(x_1, z)\otimes h(x_2)\overline{u_{n_0}(x_2)^l}\,.
$$
Indeed, on the involutive subalgebra 
$$
\cc_o:=\bigg\{\sum_{l\in F}g_l(x)z^{ln_o}\mid g_l(x)\in C(X_o),\,\,F\subset\bz\,\,\text{finite}\bigg\}\,,
$$
$\iota$ is a well-defined $*$-homomorphism. 

We want to show that $\iota$ is a positive map between the operator system $\cc_o$ and the $C^*$-algebra $\mathcal{A}_1\otimes L^\infty(X_o, \mu_o)$. To this aim, we first note that 
$\iota$ extends to $\cb\cb_o$, the $*$-algebra made of elements of the form $\sum_{l\in F}g_l(x)z^{ln_o}$, $F\subset\bz$ finite and $g_l$ bounded Borel functions on $X_o$. 

We now fix an element 
$h\in\cc_o$ such that $h=c^*c$ for some $c\in\cc$, and define $h_n:=h+1/n$. The $h_n$ belong to $\cb\cb_o$ and are strictly positive, and thus by Proposition \ref{wcrf}, 
there exists $(b_n)_n\subset\cb\cb_o$ such that $h=b_n^*b_n$. Therefore, 
$$
\iota(h)+\frac1{n}=\iota(h_n)=\iota(b_n^*b_n)=\iota(b_n)^*\iota(b_n)\in V\,,
$$
the convex cone of the positive elements of the $C^*$-algebra $\mathcal{A}_1\otimes L^\infty(X_o, \mu_o)$, which is closed by \cite{T}, Theorem I.6.1. 

By taking the limit on $n$, we easily deduce that $\iota:\cc_o\to \mathcal{A}_1\otimes L^\infty(X_o, \mu_o)$ is a positive map, and thus
we can apply \cite{P}, Proposition 2.1, to conclude that $\iota$ is bounded on $\cc_o$. Therefore, it extends to a bounded map on the whole 
$C(X_o\times\bt)^{\b_{n_o}}=\overline{\cc_o}$ which will be also a 
$*$-homomorphism denoted again by $\iota$ with an abuse of notation. 
The claim is now proved by noticing that $T$ can be expressed as the composition of three bounded maps as 
$T= \Big({\rm id}\otimes\int\,{\bf\cdot}\,\di\m_o\Big)\circ \iota\circ\ce_{n_o}$.

To end the proof, it is enough to verify the $\F_{\th_o,f}$ invariance on the total set of generators of the form $h(x,z)=g(x)z^n$ of $C(X_o\times\bt)$, the proof of the
module-map properties being similar. 

Indeed, if $n\neq0$ is not a multiple of $n_o$, then $T(h)=0=T(h\circ\F_{\th_o,f})$. If instead $n=ln_o$, then
\begin{align*}
T(h\circ\F_{\th_o,f})=&a_{l}\int_{X_o} g(\theta_o(x))f(x)^{ln_0}{u_{n_0}(x)}^{-l}{\rm d}\mu_o\\
=&a_{l}\int_{X_o}g(\theta_o(x))u_{ln_0}(\theta_o(x))^{-l}{\rm d}\mu_o\\
=&a_{l}\int_{X_o}g(x)u_{ln_0}(x)^{-l}{\rm d}\mu_o
=T(h).
\end{align*}
\end{proof}
\begin{rem}
\label{rcfuce}
Note that, when there are only continuous solutions of the cohomological equation ({\it i.e.} $k_o=1$), the map $\s\circ\r_1\circ T$, $\s$ given in \eqref{sigma},  yields an invariant conditional expectation onto the fixed-point subalgebra.
By \cite{DFR}, Theorem 10.7, and \cite{AD}, Theorem 3.2, it is in fact the unique invariant conditional expectation on the fixed-point subalgebra. 

When instead there are no nontrivial solutions of the cohomological equation ({\it i.e.} $n_o=0$), it is easily checked that the map $T$ yields the state on $C(X_o\times\bt)$ corresponding to $\int_{X_o\times\bt}\,{\bf\cdot}\,\di\m_o\times \di\th/2\pi$ which is the unique invariant measure under the action of $\F_{\th_o,f}$.
\end{rem}
The following ought to be known. Nevertheless, we include a sketched proof for convenience and establish some notation.
Given a square matrix $C=(c_{i,j})_{i,j=1}^k\in\bm_k(\mathbb{C})$, for $0\leq l<k-1$ its $l$-diagonal is the set of the entries
$\{c_{1, l+1}, c_{2, l+2}, \ldots, c_{k-l,k}\}$, and its $l$-trace is the number ${\rm tr}_l(C):=\sum_{j=1}^{k-l}c_{j, l+j}$.
Note that ${\rm tr}_0$ is the usual trace $\tr$ of $\bm_k(\bc)$.

Denote by $\cb_k\subset C(\bt)$ the $C^*$-algebra generated by all powers of the function $z^k$. For each $A\in M_k(\mathbb{C})_+$ with $\tr(A)=1$ and 
$x\in \pi_k(C(\bt))\subset M_k(C(\bt))$, where $\pi_k: C(\bt)\rightarrow M_k(C(\bt))$ is the
$^*$-monomorphism considered in Example \ref{exmat}, set
\begin{equation}
\label{attm}
\widetilde{F}_{A}(x):=
\begin{pmatrix}
\tr(Ax)& 0 & \cdots & 0 \\
0 & \tr(Ax) & \cdots & 0 \\
\vdots  & \vdots  & \ddots & \vdots  \\
0 & \cdots & 0 & \tr(Ax)
\end{pmatrix}
\end{equation}

By easy computations, one can verify that $\widetilde{F}_{A}$
is a conditional expectation from $\pi_k(C(\bt))$
onto $\pi_k(\mathcal{B}_k)$.
\begin{lem}
\label{caleb}
For any $A\in M_k(\mathbb{C})_+$ with $\tr(A)=1$, the map $F_{A}:=\pi_k^{-1}\circ \widetilde{F}_{A}\circ \pi_k$ provides a conditional expectation of $C(\bt)$ onto $\mathcal{B}_k$. Moreover, $F_{A}= F_{B}$ if and only if 
${\rm tr}_l(A)={\rm tr}_l(B)$ for every $l=1, \ldots, k-1$.
\end{lem}
\begin{proof}
That  $F_{A}$ is a conditional expectation is a straightforward consequence of
its definition, for $\widetilde{F}_{A}$ is a conditional expectation and $F_{A}$ is obtained out of the latter
by conjugation by the $*$-isomorphism $\pi_k$.
 
As for the second part of the statement, given two positive matrices $A=(a_{i,j}), B=(b_{i,j})\in \bm_k(\bc)$ with ${\rm tr}_0(A)={\rm tr}_0(B)=1$, one has 
$\widetilde{F}_{A}= \widetilde{F}_{B}$ if and only
if ${\rm tr}_0(Ax)={\rm tr}_0(Bx)$ for any $x\in C^*(U)$, which is the same as 
${\rm tr}_0(AU^l)={\rm tr}_0(BU^l)$ for every $l\in\bz$. 
It is actually enough to consider only positive values of the integer $l$ thanks to the equality
$$
{\rm tr}_0(AU^{-l})={\rm tr}_0((U^lA)^*)=\overline{{\rm tr}_0(U^lA)}=\overline{{\rm tr}_0(AU^l)}\,,
$$ 
which holds for every $l\in\bn$.
For $l=1, 2, \ldots, k-1$, the conditions ${\rm tr}_0(AU^l)={\rm tr}_0(BU^l)$ can be written explicitly as
\begin{equation}
\label{conditions}
z{\rm tr}_{k-l}(A)+{\rm tr}_l(A)=z{\rm tr}_{k-l}(B)+{\rm tr}_l(B), \,\textrm{for any}\, z\in\bt.
\end{equation}

If ${\rm tr}_l(A)={\rm tr}_l(B)$ for every $l=1, 2, \ldots, k-1$ then the equalities \eqref{conditions}
are certainly satisfied. Conversely, rewrite \eqref{conditions}
as $z({\rm tr}_{k-l}(A)-{\rm tr}_{k-l}(B))={\rm tr}_l(B)-{\rm tr}_l(A)$, for any $z\in\bt$. Since the r.h.s. of the last equality
does not depend on $z$, the only possibility is ${\rm tr}_{k-l}(A)-{\rm tr}_{k-l}(B)=0$, hence ${\rm tr}_l(A)={\rm tr}_l(B)$ for
any $l=1, 2, ..., k-1$.

Finally if $l\geq k$, we can rewrite $l$ as $l=mk+l'$, for some $m\in\bn$ and $l'=0, 1,\ldots, k-1$. But then
${\rm tr}_0(AU^l)={\rm tr}_0(AU^{mk}U^{l'})={\rm tr}_0(U^{mk}AU^{l'})=z^m{\rm tr}_0(AU^{l'})=z^m (z{\rm tr}_{k-l'}(A)+{\rm tr}_{l'}(A))$ (or simply
$z^m{\rm tr}_0(A)$ when $l'=0$), which means
${\rm tr}_0(AU^l)={\rm tr}_0(BU^l)$ is still satisfied. This ends the proof.
\end{proof}
\begin{prop}
\label{manycond}
Let $m_o=n_ok_o$ be different from $0$. Then, for $A\in M_{k_o}(\mathbb{C})_+$ with $\tr(A)=1$, the map
\begin{equation}
\label{uncd}
E_A:=\s\circ F_A\circ\r_1\circ T
\end{equation}
provides an invariant conditional expectation of $C(X_o\times \bt)$ onto the fixed-point subalgebra $C(X_o\times \bt)^{\F_{\th_o,f}}$.

Moreover, $E_A= E_B$ if and only if 
${\rm tr}_l(A)={\rm tr}_l(B)$ for every $l=1, \ldots, k-1$, where ${\rm tr}_l$ is the $l$-trace.
\end{prop}
\begin{proof}
The fact that $E_A$ is an invariant conditional expectation follows immediately from the properties of $T$ stated in Propostion \ref{quasiexp}, by taking into account the identification
$C(\bt)\sim C(X_o\times \bt)^{\F_{\th_o,f}}$ given in \eqref{sigma}.

Since the range of $T$ is dense in $\mathcal{A}_1$, the equality $E_A=E_B$ holds if and only if $F_A=F_B$, and 
the second statement  follows directly from Lemma \ref{caleb}.
\end{proof}
As noticed above ({\it cf.} Remark \ref{rcfuce}), in the case $k_o=1$ that is when there are only continuous solutions of the cohomological equations \eqref{sifo}, \eqref{uncd} provides the unique $\F_{\th_o,f}$-invariant conditional expectation of $C(X_o\times\bt)$ onto the fixed-point subalgebra.

\section{invariant conditional expectations and ergodicity w.r.t. the fixed-point subalgebra}

We start the present section to prove the main result of the paper, that is to answer in positive Question 3.4 in \cite{AD} for the (classical) Anzai skew-product. For such a purpose, we start with the following
\begin{rem}
\label{kap0}
If there are no non trivial continuous solutions to the cohomological equations, then the fixed-point subalgebra
$C(X_o\times\mathbb{T})^{\Phi_{\theta_o, f}}$ is trivial. Therefore, 
a $\Phi_{\theta_o, f}$-invariant conditional expectation is simply described by a $\Phi_{\theta_o, f}$-invariant state. 

Since there are non trivial measurable non-continuous
solutions, the system cannot be uniquely ergodic and thus there exist infinitely many $\Phi_{\theta_o, f}$-invariant states.
For instance, a family of such states is obtained by considering the composition $\varphi\circ T$, where $\varphi$ is any state on
$\mathcal{A}_1$ and $T$ is described in Proposition \ref{quasiexp}.
\end{rem}
We recall that, for the definition of ergodicity w.r.t the fixed-point subalgebra, we are adopting Definition \ref{ADFM} which turns out to be equivalent to the conditions (i)-(v) listed in \cite{AD}, Theorem 3.2 (see also \cite{FM3}, Theorem 2.1).
\begin{thm}
\label{abadyk}
Let $(X_o,\th_o,\m_o)$ be a uniquely ergodic dynamical system, together with the Anzai skew product $\F_{\th_o,f}$ associated to $f\in C(X_o;\bt)$. Then the following are equivalent:
\begin{itemize}
\item[(i)] $\F_{\th_o,f}$ is uniquely ergodic w.r.t. the fixed-point subalgebra,
\item[(ii)] there exists only one $\F_{\th_o,f}$-invariant conditional expectation onto the fixed-point subalgebra.
\end{itemize}
\end{thm}
\begin{proof}
We start by recalling Proposition \ref{nzmz} describing the structure of the solutions of the solutions of the cohomological equations \eqref{sifo} and \eqref{sifo1}. In the sequel we adopt the notations used therein. 

The case when there exists no nontrivial solutions (but the constant ones) corresponds to $n_o=0$ for which the result holds true ({\it e.g.} Remark \ref{uefpsb}). Therefore, we assume that \eqref{sifo} admits nontrivial solutions for some $n\in\bz\smallsetminus\{0\}$, which corresponds $n_o>0$.

Since the implication (i)$\Rightarrow$(ii) holds in general, it remains to prove the implication (ii)$\Rightarrow$(i) for a nontrivial fixed-point subalgebra $C(X_o\times\mathbb{T})^{\Phi_{\theta_o, f}}$.
If (i) does not hold, by Theorem 10.7 in \cite{DFR} there must exist measurable non-continuous solutions
of the cohomological equations, which means Proposition \ref{manycond} applies providing plenty of conditional expectations, and (ii) does not hold either.
\end{proof}
\begin {rem}
\label{3eq}
As noticed before ({\it cf.} \cite{DFR}, Theorem 10.7), the (equivalent) conditions (i) and (ii) are also equivalent to
\begin{itemize}
\item[(iii)] the cohomological equations \eqref{sifo} admit only continuous solutions, that is if $g$ satisfies \eqref{sifo} for some $n\in\bz$, then $g=\pi_{\m_o}(G)$ where $G$ satisfies \eqref{sifo1} for the same $n$ or, in other words, $G\in C(\bt)$ and $g=G$, $\m_o$-a.e.\,\,. 
\end{itemize}
\end{rem}
\noindent
By Proposition \eqref{nzmz} (and with the notations used therein), (iii) corresponds either to $n_o=0$, the uniquely ergodic case, or $n_o>0$ and $k_o=1$, the uniquely ergodic cases with nontrivial fixed-point subalgebra.

\medskip

We now pass to study some properties of the set of all invariant conditional expectations which allow us to characterise unique ergodicity w.r.t. the fixed-point subalgebra. With $F_A$ and $\ce_n$, we refer to the maps in Lemma \ref{caleb} and \eqref{perce}, respectively.
\begin{prop}
\label{domination}
Suppose $m_o>0$. Then the $\Phi_{\th_o,f}$-invariant conditional expectation $E_{\frac{1}{k_o}\idd}=\s\circ F_{\frac{1}{k_o}\idd}\circ\r_1\circ T$ associated with $\frac{1}{k_o}\idd$ satisfies
\begin{equation}\label{absorption}
E_{\frac{1}{k_o}\idd}\circ \mathcal{E}_{m_0}=E_{\frac{1}{k_o}\idd}.
\end{equation}

In addition, every $\Phi_{\th_o,f}$-invariant conditional expectation $E$ from $C(X_o\times\bt)$ onto $C(X_o\times\bt)^{\Phi_{\th_o,f}}$  satisfies
\begin{equation}
\label{PP}
E(a)\leq  m_o E_{\frac{1}{k_o}\idd}(a)\,,\quad a\in C(X_o\times\bt)_+\,.
\end{equation}
\end{prop}
\begin{proof}
For the first part of the statement it is enough to show
$$
E_{\frac{1}{k_o}\idd}\circ \mathcal{E}_{m_o}(h\cdot \chi_l)=E_{\frac{1}{k_o}\idd}(h\cdot \chi_l)
$$
for $l\in\mathbb{Z}$ and  $h\in C(X_o)$, where $\chi_l(z):=z^l$, $z\in\bt$. 

Note that, if $l$ is not a multiple of $m_o$, then by the definition of $\mathcal{E}_{m_o}$, the l.h.s. is zero. The same can be said for the r.h.s. since for $l$ not a multiple of $n_o$, by the definition of the map $T$, one has $T(h \cdot \chi_l)=0$ and thus $E_{\frac{1}{k_o}\idd}(h\cdot \chi_l)=0$. 

We now handle the case when $l$ is a multiple of $n_o$ but not of $m_o$.
From the equality $T(h\cdot \chi_l)=\omega_o(h\cdot \bar u_{l}) [u_l \chi_l]_\mu$, we see that 
$$
E_{\frac{1}{k_o}\idd}(h\cdot \chi_l)=\s\circ F_{\frac{1}{k_o}\idd}\circ\r_1\circ T(h\cdot \chi_l)=\omega_o(h\cdot \bar u_{l})\,\s\circ F_{\frac{1}{k_o}\idd}\circ\r_1 ([u_l \chi_l]_\mu)\,.
$$
But then, $E_{\frac{1}{k_o}\idd}(h\cdot \chi_l)=0$ because $F_{\frac{1}{k_o}\idd}\circ\r_1 ([u_l \chi_l]_\mu)=
F_{\frac{1}{k_o}\idd}(\chi_{\frac{l}{n_o}})=0$.\\

Finally, if $l$ is a multiple of $m_0$, then $\mathcal{E}_{m_0}(h\chi_{l})=h\chi_l$
for any $h\in C(X_o)$, and the sought equality follows.

For the second statement, first note that for every invariant conditional expectation $E$ from $C(X_o\times\bt)$ onto the fixed-point subalgebra $C(X_o\times\bt)^{\Phi_f}$, one has
$$
E_{\frac{1}{k_o}\idd}=E_{\frac{1}{k_o}\idd}\circ \mathcal{E}_{m_o}=E\circ \mathcal{E}_{m_o}\,.
$$

Indeed, by density, linearity and what we saw above, it is enough to verify the equality
only on functions $h$ of the form $h(x, z)=\ell(x)z^{lm_o}$,
where $\ell\in C(X_o)$ and $l\in\bz$. 

Recalling the definition of the functions $a_n$ as $a_{n}(x, z):=\big(u_{n_0}(x)z^{n_o}\big)^n$,
one has
\begin{align*}
\left(E\circ \mathcal{E}_{m_o}(h)\right)(x, z)=&(E(h))(x, z)=\left(E(\ell\overline{u_{lk_o}}a_{lk_o})\right)(x, z)\\
=&\bigg(\frac{1}{n}\sum_{j=0}^{n-1}E(h\circ\Phi_{\theta_o, f}^j)\bigg)(x, z)\\
=&a_{lk_o}(x, z)
E\bigg(\frac{1}{n}\sum_{j=0}^{n-1} \ell(\theta^j(x))\overline{u_{lm_o}(\theta^j(x))}\bigg)\\
=&a_{lk_o}(x, z)\om_o(\ell\overline{u_{lm_o}}),
\end{align*}
where the last equality has been obtained by exploting the unique ergodicity of $\theta_o$, whose unique invariant state is
$\om_o$.

Now, from the equality $\mathcal{E}_{m_o}=\frac{1}{m_o}\sum_{l=0}^{m_o-1} \beta_{m_o}^l$, we see  that  
$$
E_{\frac{1}{k_o}\idd}= \frac{1}{m_o}E+\frac{1}{m_o}\sum_{l=1}^{m_o-1}E\circ\beta_{m_o}^l\,,
$$ 
hence
$E_{\frac{1}{k_o}\idd}(a)\geq  \frac{1}{m_o}E(a)$ for all positive functions $a$ in $C(X_o\times \bt)$, as stated.
\end{proof}
Consider the convex set $\ck:=\big\{E\mid  E\,\textrm{is a}\,\Phi_{\th_o,f}$-invariant conditional expectation of $C(X_o\times\bt)$ onto $C(X_o\times\bt)^{\Phi_{\th_o,f}}\big\}$. In the topologically ergodic situation ({\it i.e.} $C(X_o\times\bt)^{\Phi_{\th_o,f}}=\bc\idd$), any such  conditional expectation is associated with a $\Phi_{\th_o,f}$-invariant state $\f$: $E\equiv E_\f:=\f(\,{\bf\cdot}\,)\idd$.

It was proved in \cite{Fu} (see also \cite{DFGR, DFR} for the noncommutative cases) that\\

\noindent
$\big(X_o\times\bt, \Phi_{\th_o,f}\big)$ topologically ergodic $\&$ $E_{\f_\m}$ extremal $\Rightarrow$ $\big(X_o\times\bt, \Phi_{\th_o,f}\big)$ uniquely ergodic, or equivalently $\ck$ is a singleton.

\bigskip

We want to extend this result to the non-topologically ergodic situation corresponding to $m_o>0$, where $E_{\f_\m}$ is replaced by $E_{\frac{1}{k_o}\idd}$. We set
$$
E_{\rm can}:=\left\{
\begin{array}{ll}
E_{\f_\m}& \text{if}\,\, m_o=0\,, \\
E_{\frac{1}{k_o}\idd} &\text{if}\,\, m_o>0\,. \\
\end{array}
\right.
$$
\begin{thm}
For the Anzai skew-product $C^*$-dynamical system $(X_o\times\bt, \Phi_{\theta_o, f})$, the equivalent conditions ${\rm(i), (ii)}$ in Theorem \ref{abadyk} and ${\rm(iii)}$ in Remark \ref{3eq} are equivalent to
\begin{itemize}
\item[(iv)] $E_{\text{can}}$ is extreme among all $\Phi_{\theta_o, f}$-invariant conditional expectations of $C(X_o\times\bt)$ onto $C(X_o\times\bt)^{\Phi_{\theta_o, f}}$.
\end{itemize}
\end{thm}
\begin{proof}
 We only need to focus on the case $m_o>1$, and thus $E_{\rm can}=E_{\frac{1}{k_o}\idd}$. Indeed,
 the assertion is certainly true when $m_o=0$, \emph{i.e.} in the topologically ergodic case, whereas 
 the case $m_o=1$ corresponds to a uniquely ergodic system w.r.t. the fixed-point subalgebra, see \emph{e.g.}
Remark \ref{k_o}.

The implication (ii)$\Rightarrow$(iv) is entirely obvious.  We limit ourselves to the reverse implication. 
 Suppose that $E_{\frac{1}{k_o}\idd}$ is extreme among
all invariant conditional expectations. If there existed an invariant conditional
expectation $E$ different from $E_{\frac{1}{k_o}\idd}$ then, thanks to
Proposition \ref{domination}, we would have $E(a^*a)\leq m_o E_{\frac{1}{k_o}\idd}(a^*a)$, for every $a\in C(X_o\times \bt)$.
But then $F:=\frac{1}{m_o-1}(m_o E_{\frac{1}{k_o}\idd}-E)$ would be an invariant conditional
expectation as well, and $E_{\frac{1}{k_o}\idd}$ could be written as a proper convex combination
$E_{\frac{1}{k_o}\idd}=\frac{1}{m_o}E+\frac{m_o-1}{m_o}F$, which is
a contradiction.
\end{proof}

\section{a simple example}

For the convenience of the reader, we briefly revisit Example 11.3 of \cite{DFR} and compute the conditional expectations described in Proposition \ref{manycond}.

Indeed, let $\bz_\infty$ be the one-point compactification of the integers $\bz$, together with the homeomorphism 
$\th_o:\bz_\infty\rightarrow\bz_\infty$ given by
$$
\th_o(l):=\left\{\begin{array}{ll}
                     l+1 & \!\!\text{if}\,\, l\in \bz\,, \\
                      \infty & \!\!\text{if}\,\, l=\infty\,.
                    \end{array}
                    \right.
$$

The dynamical system $(\bz_\infty,\th)$ is uniquely ergodic (but neither minimal, nor strictly ergodic) with the unique invariant measure
$$
\m_o(f):=f(\infty)\,,\quad f\in C(\bz_\infty)\,.
$$

We also note that 
$$
C(\bz_\infty)^{\th_o}:=\{f\in C(\bz_\infty)\mid f\circ\th=f\}=\bc 1\sim L^\infty(\bz_\infty,\m_o)\,,
$$
with $1$ being the function identically equal to one.

For $f\in C(\bz_\infty;\bt)$, we can associate the process on the torus $\F_{\th_o,f}:\bz_\infty\times\bt\rightarrow\bz_\infty\times\bt$ given by $(\F_{\th_o,f})(l,z)=(\th_o(l), f(l)z)$. We now particularise the situation for 
$$
f(l):=\left\{\begin{array}{ll}
                    -1 & \!\!\text{if}\,\, l=0\,, \\
                      1 & \!\!\text{otherwise}\,.
                    \end{array}
                    \right.
$$

It is immediate to show that for each $n\in\bz$, 
$$
\text{const}f(l)^n=\text{const}\,,\,\,\mu_o\,\text{-}\,\textrm{a.e.}\,,
$$
and therefore the \eqref{sifo} admit nontrivial solutions for each $n\in\bz$.

On the other hand, if the two-sided sequence $(g(l))_{l\in\bz}$ satisfies \eqref{sifo1} for some $n$, then
$$
g(l):=\left\{\begin{array}{ll}
                    g(0)(-1)^{-n} & \!\!\text{if}\,\, l>0\,, \\
                     g(0) & \!\!\text{if}\,\, l\leq0\,,
                    \end{array}
                    \right.
$$
Imposing continuity, we are led to
$$
g(0)=\lim_{l\to-\infty}g(l)=\lim_{l\to+\infty}g(l)=(-1)^{-n}g(0)\,.
$$

Therefore, the \eqref{sifo1} admit nontrivial solutions if $n$ is even, and correspondingly $n_o=1$, $m_o=2=k_o$. The fixed-point subalgebra is linearly generated by elements of the form $g(l,z)=az^{2n}$, $n\in\bz$. 

For $h\in C(\bz_\infty,\bt)$, with 
$$
h(l,z)=\sum_{n\in\bz}\big(h_{2n}(l)z^{2n}+h_{2n+1}(l)z^{2n+1}\big)\,,
$$
we easily obtain
$$
T(h)(z)=h(\infty,z)=\sum_{n\in\bz}\big(h_{2n}(\infty)z^{2n}+h_{2n+1}(\infty)z^{2n+1}\big)\,.
$$

After some computations, for $A\in\bm_2(\bc)$, positive and normalised, we deduce
$$
E_A(h)(l,z)=\sum_{n\in\bz}\big(h_{2n}(\infty)+h_{2n+1}(\infty)(a_{12}+a_{21}z^2)\big)z^{2n}\,,
$$
and thus $E_{\frac12\idd}(h)(l,z)=\sum_{n\in\bz}h_{2n}(\infty)z^{2n}$.

We now sketch some computations in order to verify \eqref{PP} for $\widetilde{E}=E_A$, and for the simplest nontrivial situation $h=\overline{g}g$ where $g(l,z)=g_0(l)+g_1(l)z$ (see also Proposition \ref{wcrf}). For such a purpose, we first note that,
if $A=\begin{pmatrix}
\l& a_{12} \\
\overline{a_{12}} & 1-\l 
\end{pmatrix}$ and $0\leq\l\leq1$, $|a_{12}|\leq a(\l)\leq1/2$, where $a(\l)$ is the greatest value that $|a_{12}|$ can assume.
On the other hand, $E_A(|g|^2)\leq E_{\frac12\idd}(|g|^2)$ leads to $4|g_0(\infty)||g_1(\infty)||a_{12}|\leq|g_0(\infty)|^2+|g_1(\infty)|^2$,
which is automatically satisfied if $|g_0(\infty)|$ or $|g_1(\infty)|$ is $0$. If indeed $|g_0(\infty)||g_1(\infty)|\neq0$,
$|a_{12}|\leq\frac{|g_0(\infty)|^2+|g_1(\infty)|^2}{4|g_0(\infty)||g_1(\infty)|}$
with the r.h.s. always greater than $1/2$, and the assertion follows.

\section*{Acknowledgements}

S.D.V. and F.F. acknowledge the \lq\lq MIUR Excellence Department Project'' awarded to the Department of Mathematics, University of Rome Tor Vergata, CUP E83C18000100006.

\end{document}